\documentclass[11pt,a4paper,reqno]{amsart}
\usepackage{graphicx} 
\usepackage{epsfig}
 \newtheorem{thm}{Theorem}[section]
 
 \newtheorem{lem}[thm]{Lemma}
 \newtheorem{prop}[thm]{Proposition}
 \theoremstyle{definition}
 \newtheorem{defn}[thm]{Definition}
 \theoremstyle{remark}
 
\usepackage[american]{babel}
\usepackage{csquotes}
\usepackage{amsmath,amssymb}
\usepackage{xypic}
\usepackage{booktabs}
\usepackage{microtype}
\newcommand{\tens}[1]{\ensuremath{\langle #1 \rangle}}
\newcommand{\cc}{\mathbb C} 
\newcommand{\cO}{\mathcal O} 
\newcommand{\rr}{\mathbb R}
\newcommand{\oR}{\overline{R}}
\newcommand{\uR}{\underline{R}}
\let\bs\boldsymbol
\let\mi\mathit
\usepackage{xspace}

\DeclareMathOperator{\sgn}{sgn}
\DeclareMathOperator{\Sym}{Sym}
\DeclareMathOperator{\supp}{supp}
\newtheorem*{ex}{Example}
\theoremstyle{plain}
\newtheorem*{obss}{Observations}
\usepackage[norelsize,ruled,noend,scleft]{algorithm2e}
\SetCommentSty{}
\SetKwComment{Comment}{}{}
\SetAlFnt{\small}
\usepackage{todonotes}

\makeatletter
  \providecommand\@dotsep{5}
\makeatother

\begin{document}
\setcounter{page}{1}
\begin{flushleft}
\scriptsize Appl. Comput. Math., V.xx, N.xx, 20xx, pp.xx-xx
\end{flushleft}
\bigskip
\bigskip
\title[Hart et al.: A Fast Search Algorithm for \tens{m,m,m} TPP Triples]%
{A Fast Search Algorithm for $\boldsymbol{\tens{m,m,m}}$ Triple Product Property Triples and 
an Application for $\boldsymbol{5\times 5}$ Matrix Multiplication}
\author[Appl. Comput. Math., V.xx, N.xx,  20xx]%
{Sarah Hart$^1$, Ivo Hedtke$^2$, Matthias M\"{u}ller-Hannemann$^2$ and Sandeep Murthy$^1$}
\thanks{$^1$Department of Economics, Mathematics \& Statistics,
Birkbeck, University of London,
Malet Street, London,
\\ \indent\,\,\,WC1E 7HX, United Kingdom.
\\ \indent\,\,\,e-mail: s.hart@bbk.ac.uk, s.murthy@ems.bbk.ac.uk\\
\indent $^2$Institute of Computer Science, Martin-Luther-University of Halle-Wittenberg, Von-Seckendorff-Platz~1,
\\ \indent\,\,\,06120~Halle~(Saale), Germany.
\\ \indent\,\,\,e-mail: hedtke@informatik.uni-halle.de, muellerh@informatik.uni-halle.de
\\ \indent
  \em \,\,\,Manuscript received xx}

\begin{abstract}
We present a new fast search algorithm for \tens{m,m,m} Triple Product Property (TPP) triples 
as defined by Cohn and Umans in 2003. The new algorithm achieves a speed-up factor of 40 up to 194 in 
comparison to the best known search algorithm. With a parallelized version of the new algorithm we 
are able to search for TPP triples in groups up to order 55.\par
As an application we identify a list of groups that would realize $5\times 5$ matrix 
multiplication with under 100 resp.\ 125 scalar multiplications (the best known upper bound by 
Makarov 1987 resp.\ the trivial upper bound) if they contain a $\tens{5,5,5}$ TPP triple. With 
our new algorithm we show that no group can realize $5\times 5$ matrix multiplication 
better than Makarov's algorithm.


\bigskip
\noindent AMS Subject Classification (MSC2010): Primary 20-04,
68Q25, Secondary 20D60, 68Q17, 68R05

\end{abstract}

\maketitle

\smallskip
\section{Introduction}

\subsection{A Very Short History of Fast Matrix Multiplication}
The naive algorithm for matrix multiplication is an $\cO(n^3)$ algorithm.
From Strassen \cite{Strassen1969} we know that there is an $\cO(n^{2.81})$ algorithm for
this problem.
One of the most famous results is an $\cO(n^{2.3755})$
algorithm from Coppersmith and Winograd
\cite{Coppersmith1990}.
Recently, Williams \cite{Williams2012} found an algorithm with
$\mathcal O (n^{2.3727})$
run-time based on the work of Stothers \cite{Stothers2010}.
Let $M(n)$ denote the number of field operations in
characteristic 0 required to multiply two $(n\times n)$ matrices. Then we call
$\omega := \inf \{r \in \rr  :  M(n)=\cO(n^r)\}$
the exponent of matrix multiplication. Details about
the complexity of matrix multiplication and the exponent $\omega$ can be found
in \cite{Buergisser1997}. 

\subsection{A Very Short History of Small Matrix Multiplication}
The naive algorithm uses $n^3$ multiplications and $n^3-n^2$ additions
to compute the product of two $n \times n$ matrices.
The famous result $\cO(n^{2.81})$ is based on an algorithm that can
compute the product of two $2\times 2$ matrices with only $7$
multiplications.
Winograd \cite{Winograd1971} proved that the minimum number of multiplications required in 
this case is $7$. The exact number $R(n)$ of required multiplications to compute the product 
of two $n \times n$ matrices is \emph{not} known for $n>2$. There are known upper bounds for 
some cases. Table~\ref{tab:upToFive} lists the known upper bounds for $R(n)$ up to $n=5$. 
Tables for up to $n=30$ can be found in \cite[Section~4]{Drevet2011}.
Hedtke and Murthy proved in \cite[Theorem~7.3]{HedtkeMurthy2011}
that the
group-theoretic framework (discussed in Subsection~\ref{subsec:CohnUmans})
is not able to produce better 
bounds for $R(3)$ and $R(4)$.

\subsection{Bilinear Complexity}
Later we will use the concept of bilinear complexity to connect group-theoretic arguments with 
the complexity of matrix multiplication.

\begin{defn}[Rank]\cite[Chapter~14 and Definition~14.7]{Buergisser1997}
Let $k$ be a field and $U,V,W$ finite dimensional $k$-vector spaces. Let
$\eta \colon U \times V \to W$ be a $k$-bilinear map. For
$i \in \{1,\ldots,r\}$ let $f_i\in U^*$, $g_i\in V^*$ (dual spaces of $U$ and $V$ resp.\ over
$k$) and $w_i\in W$ such that
\[
\eta(u,v)=\sum_{i=1}^r f_i(u)g_i(v)w_i
\]
for all $u\in U$ and $v\in V$. Then $(f_1,g_1,w_1; \ldots ; f_r,g_r,w_r)$ is called a
\emph{$k$-bilinear algorithm of length $r$ for $\eta$}, or simply a \emph{bilinear algorithm} 
when $k$ is fixed. The minimal
length of all bilinear algorithms for $\eta$ is called the
\emph{rank} $R(\eta)$ of $\eta$. Let $A$ be a $k$-algebra.
The \emph{rank} $R(A)$ of $A$ is defined as the rank of its bilinear multiplication map.
\end{defn}

\begin{table}
\centering
\begin{tabular}{lrl}
\toprule
$n\times n$ & upper bound for $R(n)$ & algorithm\\
\midrule 
$2\times 2$ & 7 & Strassen \cite{Strassen1969}\\
$3\times 3$ & 23 & Laderman \cite{Laderman1976}\\
$4\times 4$ & 49 & Strassen \cite{Strassen1969}\\
$5\times 5$ & 100 & Makarov \cite{Makarov1987}\\
\bottomrule
\end{tabular}
\caption{Upper bounds for $R(2)$, $R(3)$, $R(4)$ and $R(5)$.}
\label{tab:upToFive}
\end{table}

\begin{defn}[Restriction of a bilinear map]\cite[Definition~14.27]{Buergisser1997}
Let $\phi\colon U \times V \to W$ and $\phi'\colon U' \times V' \to W'$ be $k$-bilinear maps. A \emph{$k$-restriction}, or simply a \emph{restriction} (when $k$ is fixed), of $\phi'$ to $\phi$ is a triple $(\sigma,\tau,\zeta')$ of linear maps $\sigma\colon U\to U'$, $\tau\colon V\to V'$ and $\zeta'\colon W' \to W$ such that $\phi = \zeta'\circ \phi' \circ (\sigma \times \tau)$:
\[
\xymatrix{
U\times V \ar[d]_{\sigma\times\tau}\ar @{} [dr] |{\copyright} \ar[r]^-{\phi} & W \\
U'\times V' \ar[r]_-{\phi'}        & W'\ar[u]_{\zeta'}        }
\]
We write $\phi \leq \phi'$ if there exists a restriction of $\phi'$ to $\phi$.
\end{defn}

\subsection{The Group-Theoretic Approach of Cohn and Umans}
\label{subsec:CohnUmans}

In 2003 Cohn and Umans introduced in \cite{Cohn2003} 
a
group-theoretic approach to fast matrix multiplication. The main idea of their framework 
is to embed the matrix
multiplication over a ring $R$ into the group ring $R[G]$ of a
group $G$. A group $G$ admits such an embedding if there are subsets $S$, $T$ and $U$ of
$G$ which satisfy the so-called \emph{Triple Product Property}.  

\begin{defn}[right quotient]
Let $G$ be a group and $X$ be a nonempty subset of $G$.
The \emph{right quotient} $Q(X)$ of $X$ is defined by
$Q(X):=\{xy^{-1} : x,y \in X\}$.
\end{defn}

\begin{defn}[Triple Product Property]\label{def:TPP}
We say that the nonempty subsets $S$, $T$ and $U$ of a group $G$ satisfy the
\emph{Triple Product Property} (TPP) if for $s\in Q(S)$, $t\in Q(T)$ and
$u \in Q(U)$, $stu=1$ holds if and only if $s=t=u=1$.
\end{defn}

Let $k$ be a field. By $\tens{n,p,m}_k$ we denote the bilinear map
$k^{n\times p}\times k^{p\times m} \to k^{n\times m}$,
$(A,B)\mapsto AB$ describing the
multiplication of $n \times p$ by $p \times m$ matrices over $k$. When $k$ is fixed, we simply write $\tens{n,p,m}$.
\emph{Unless otherwise stated we will only work over $k=\cc$ in the entire paper.} We say that a group
$G$ \emph{realizes} $\tens{n,p,m}$ if there are subsets $S,T,U\subseteq G$ of sizes
$|S|=n$, $|T|=p$ and $|U|=m$, which satisfy the TPP. In this case we call $(S,T,U)$ a
\emph{TPP triple} of $G$, and we define its \emph{size} to be $npm$.

\begin{defn}[TPP capacity]
We define the \emph{TPP capacity} $\beta (G)$ of a group $G$ as
$\beta(G) := \max\{npm : G\text{ realizes }\tens{n,p,m}\}$.
\end{defn}

Let us now focus on the embedding of the matrix
multiplication into $\cc [G]$. Let $G$ realize $\tens{n,p,m}$ through the subsets $S$, $T$
and $U$. Let $A$ be an $n\times p$ and $B$ be a $p\times m$ matrix. We index the entries
of $A$ and $B$ with the elements of $S$, $T$ and $U$ instead of numbers. Now we have
\[
(AB)_{s,u}=\sum_{t \in T} A_{s,t}B_{t,u}.
\]
Cohn and Umans showed that this is the same as the coefficient of
$s^{-1}u$ in the product
\begin{gather}
\label{eq:embedding}
\Big( \sum\nolimits_{s\in S, t\in T} A_{s,t}s^{-1}t\Big)
\Big( \sum\nolimits_{\hat t\in T,u\in U} B_{\hat t,u}\hat t^{-1}u\Big).
\end{gather}
So we can read off the matrix product from the group ring product by looking at the
coefficients of $s^{-1}u$ with $s\in S$ and $u\in U$.

\begin{defn}[$r$-character capacity]
Let $G$ be a group with the character degrees $\{d_i\}$. We define the \emph{$r$-character capacity} of $G$ as $D_r(G):=\sum_i d_i^r$.
\end{defn}

We write $R(n,p,m)$ for the rank of the bilinear map $\tens{n,p,m}$, and $R(n)$ for $R(n,n,n)$.  If $G$ realizes $\tens{n,p,m}$ then $\tens{n,p,m}\leq \cc [G]$ (see \cite[Theorem~2.3]{Cohn2003}) by the construction above and therefore $R(n,p,m) \leq R(\cc [G])=:R(G)$:
\[
\xymatrix@C=0pt{
\cc^{n\times p}\ar[d]_{\begin{minipage}{1.8cm}\scriptsize embedding\\ \eqref{eq:embedding} into $\cc [G]$\end{minipage}} & \times & \cc^{p\times m}\ar[d]\ar @{} [drrrrrrrrrrrrrrr] |{\copyright} \ar[rrrrrrrrrrrrrrr]^{\text{matrix multiplication}} &&&&&&&&&&&&&&& \cc^{n\times m}\\
\cc [G] & \times & \cc [G]\ar[rrrrrrrrrrrrrrr]_{\text{multiplication in $\cc [G]$}} &&&&&&&&&&&&&&& \cc [G]\ar[u]_{~~\begin{minipage}{2.6cm}\raggedright\scriptsize $(AB)_{s,u}$ = coeffi- cient of $s^{-1}u$ in~\eqref{eq:embedding}\end{minipage}}
}
\]
From Wedderburn's structure theorem it follows that $R(G) \leq \sum_i R(d_i)$. The exact value of $R(G)$ is known only in a few cases. So, usually we will work with the upper bound $D_3(G) \geq \sum_i R(d_i)$, which follows from the rank $d^3$ of the naive matrix multiplication algorithm for $\tens{d,d,d}$. We can now use $\beta(G)$ and $D_r(G)$ to get new bounds for $\omega$:

\begin{thm}\textup{\cite[Theorem~4.1]{Cohn2003}}
\label{thm:SpezialCohn}
If $G\neq 1$ is a finite group, then $\beta(G)^{\frac{\omega}{3}} \leq D_\omega (G)$.
\end{thm}

Finally we collect some results to improve the performance of our algorithms in the next sections.

\begin{lem}\label{lem:PermTPP}
\textup{\cite[Lemma~2.1]{Cohn2003}} Let $(S,T,U)$ be a TPP triple. Then for every permutation $\pi \in \Sym(\{S,T,U\})$ the triple $(\pi(S),\pi(T),\pi(U))$ satisfies the TPP.
\end{lem}

\begin{lem}\textup{\cite[Observation~2.1]{Neumann2011}}\label{lemm:Neumann}
Let $G$ be a group. If $(S,T,U)$ is a TPP triple of $G$, then $(dSa,dTb,dUc)$ is a TPP triple for all $a,b,c,d \in G$, too.
\end{lem}

Lemma~\ref{lemm:Neumann} is one of the most useful results about TPP triples. It allows us to restrict the search for TPP triples to sets that satisfy $1 \in S \cap T \cap U$.
\begin{defn}[Basic TPP triple]
Following Neumann \cite{Neumann2011}, we shall call a TPP triple $(S,T,U)$ with $1 \in S \cap T \cap U$
a \emph{basic} TPP triple.
\end{defn}

For that reason, we
will assume throughout that every TPP triple is a \emph{basic}
TPP triple.

\begin{lem}\textup{\cite[Observation~3.1]{Neumann2011}}
\label{neumannthm}
If $(S,T,U)$ is a TPP triple, then $|S|(|T|+|U|-1)\leq |G|$, $|T|(|S|+|U|-1)\leq |G|$ and $|U|(|S|+|T|-1)\leq |G|$.
\end{lem}

\begin{thm}\textup{\cite[Theorem~3.1]{HedtkeMurthy2011}}
\label{thm:HedtkeMurthy}
Three sets $S_1$, $S_2$ and $S_3$ form a TPP triple $(S_1,S_2,S_3)$ if and only if for all $\pi \in \Sym(3)$
\begin{gather*}
1\in S_1 \cap S_2 \cap S_3, \quad 
Q(S_{\pi_2}) \cap Q(S_{\pi_3}) = 1, \quad \text{and} \quad
Q(S_{\pi_1}) \cap Q(S_{\pi_2})Q(S_{\pi_3})=1.\end{gather*} 
\end{thm}

\subsection{The Aim of this Paper}

The second and fourth authors of this paper created what  we believe are currently the most efficient search algorithms for TPP triples \cite{HedtkeMurthy2011}. They also showed that the presented group-theoretic framework is not able to give us new and better algorithms for the multiplication of $3 \times 3$ and $4\times 4$ matrices over the complex numbers.

To attack the $5\times 5$ matrix multiplication problem we develop a new efficient search algorithm for \tens{m,m,m} (especially \tens{5,5,5}) TPP triples. For this special case of TPP triples it is faster than any other search algorithm and it can easily be parallelized to run on a supercomputer.

Even with the new algorithm, it is not feasible simply to test all
groups of order less than 100 (best known upper bound for $R(5)$) for \tens{5,5,5} triples. Therefore we develop theoretical methods to reduce the list of candidates that must
be checked. We show that the group-theoretic framework cannot give us a new upper bound for $R(5)$.

We will also produce a list of groups that could in
theory realize a nontrivial (with less than 125 scalar multiplications) multiplication algorithm for $5\times 5$ matrices. Additionally we show how it could be possible to construct a matrix multiplication algorithm from a given TPP triple.

\section{The Search Algorithm for $\tens{m,m,m}$ TPP Triples}
\noindent In this section we describe the basic idea and important implementation details for our new fast search algorithm
for $\tens{m,m,m}$ triples. The goal of the algorithm is to find
possible candidates for TPP triples $(S,T,U)$ using the following necessary and sufficient conditions:
\begin{gather}
\label{eq:AlgoCond}
1\in S\cap T \cap U \qquad \text{and} \qquad
Q(S)\cap Q(T)=Q(S)\cap U = Q(T)\cap U = 1.
\end{gather}
The second condition is a weaker formulation of the known result using
$Q(U)$ (in Theorem~\ref{thm:HedtkeMurthy}), but it is more useful in our algorithm. For each TPP candidate
that comes from the algorithm we test if it satisfies the TPP or not (e.g. with a TPP test from \cite[Section~4]{HedtkeMurthy2011}).

Let $G$ be a finite group. Let $n:=|G|-1$. Let $(g_0:=1_G,g_1,\ldots,g_n)$ be an arbitrary but fixed order of the elements of $G$. We want to find an $\tens{m,m,m}$ TPP triple $(S,T,U)$ (or possible TPP triple candidates) of subsets of $G$. For this, we will represent $S$, $T$ and $U$ via their basic binary representation:

\begin{defn}[binary representation]
If $X$ is an arbitrary subset of $G$ we write the \emph{binary representation $b_X$} of $X$ as an element of $\{0,1\}^{|G|}$, where $(b_X)_\ell = 1$ if and only if $g_\ell \in X$ and $(b_X)_\ell = 0$ otherwise ($0\leq \ell \leq n$).
\end{defn}

Because we only consider basic TPP triples, $(b_S)_0=(b_T)_0=(b_U)_0=1$, so we only need to consider the binary representations for $1\leq \ell \leq n$. We call this the \emph{basic binary representation} $b_S^*$, $b_T^*$ and $b_U^*$. We define $\supp(b^*_X):=\{i : (b^*_X)_i=1\}=\{i: i>0, ~ g_i \in X\}$ as the \emph{support} of a basic binary representation $b_X^*$. For example, if $|G|=8$ and $S=\{1,g_2,g_4,g_7\}$, then 
\begin{gather*}
\begin{array}{lcr}
b_S   &=& (1,0,1,0,1,0,0,1)\\
b_S^* &=& (0,1,0,1,0,0,1)\\
\supp(b_S^*) &=& \{2,4,7\}\rlap{.}
\end{array}
\end{gather*}
We want to sketch the basic idea behind the algorithms with a matrix representation of the possible TPP candidates. This representation is not efficient and will not be used in the algorithms itself. It is only used in this subsection to describe the method. Let $\bs C\in \{0,1\}^{3\times n}$ denote a matrix representation of a possible TPP candidate.  Each row of
\[
\bs C
=
\begin{bmatrix}
b_S^*\\ b_T^* \\ b_U^*
\end{bmatrix}
\]
is the basic binary representation of $S$, $T$, resp.\ $U$. We can describe the fundamental idea with three steps
\begin{enumerate}
\item[(S1)] The \enquote{moving 1} principle to find the next possible TPP triple candidate after a TPP test for the previous candidates fails.
\item[(S2)] The \enquote{marking the quotient} routine to realize Equation~\eqref{eq:AlgoCond}.
\item[(S3)] An efficient way to store the matrix $\bs C$ and access its entries.
\end{enumerate}

\subsection{The \enquote{moving 1} principle}
The \enquote{moving 1} principle is based on two observations and an idea:

\begin{obss}
\begin{enumerate}
\item The column sums of $\bs C$ are at most $1$.
\item We can restrict the search space for TPP triples with the condition
$
\min\big(\supp(b_S^*)\big) < \min\big(\supp(b_T^*)\big) < \min\big(\supp(b_U^*)\big)$.
\end{enumerate}
\end{obss}

\begin{proof}
\begin{enumerate}
\item If $M$ is a set with $1_G\in M$ it follows that $M\subseteq Q(M)$. Using Equation~\eqref{eq:AlgoCond}, we get that $X\cap Y=\{1\}$ for all $X\neq Y\in\{S,T,U\}$. Thus, $\supp(b_X^*)\cap \supp(b_Y^*)=\emptyset$ for all $X\neq Y\in\{S,T,U\}$. This proves the statement.
\item Follows immediately from Lemma~\ref{lem:PermTPP} and the fact that we are looking for TPP triples $(S,T,U)$ with $|S|=|T|=|U|$.\qedhere
\end{enumerate}
\end{proof}

The idea of the \enquote{moving 1} is as follows: After a TPP test fails we get the next candidate by moving the rightmost 1 in $b_U^*$ one step to the right. If this is not possible, delete the rightmost 1 in $b_U^*$ and move the new rightmost 1. Finally we add the missing 1 to a free spot (remember that the column sums of $\bs C$ are at most $1$).

If it is not possible (all 1's are at the right of $b_U^*$) to move a 1 in $b_U^*$, we delete the whole line $b_U^*$ and move a 1 in $b_T^*$. After this we rebuild a new line $b_U^*$ line from scratch using the two observations above. We do the same with line $b_S^*$ if no more moves in line $b_T^*$ are possible.

\begin{ex}
Let $G$ be group of order $9$. We are looking for $\tens{3,3,3}$ TPP triples. The initial configuration of $\bs C\in\{0,1\}^{3\times 8}$ would be
\[
\bs C = 
\begin{array}{|c|c|c|c|c|c|c|c|}
\hline
1 & 1&&&&&&\\\hline
&& 1 & 1&&&&\\\hline
&&&& 1 & 1 & \phantom{0} & \phantom{0}\\\hline
\end{array}
\qquad
\begin{array}{l}
b_S^* = (1,1,0,0,0,0,0,0)\\
b_T^* = (0,0,1,1,0,0,0,0)\\
b_U^* = (0,0,0,0,1,1,0,0)
\end{array}
\]
which means, that $S=\{1_G,g_1,g_2\}$, $T=\{1_G,g_3,g_4\}$ and $U=\{1_G,g_5,g_6\}$. Now we check, if $(S,T,U)$ satisfies the TPP. If so, we are finished. If not, we generate the next candidate by moving a $1$ in $\bs C$:
\[
\bs C = 
\begin{array}{|c|c|c|c|c|c|c|c|}
\hline
1 & 1&&&&&&\\\hline
&& 1 & 1&&&&\\\hline
&&&& 1 & \phantom{0}\makebox[0pt][l]{\color{red}$\rightarrow$} & 1 & \phantom{0}\\\hline
\end{array}
\]
Now $U=\{1_G,g_5,g_7\}$ and we check the TPP again. The procedure of the \enquote{moving $1$} continues if the TPP check fails:
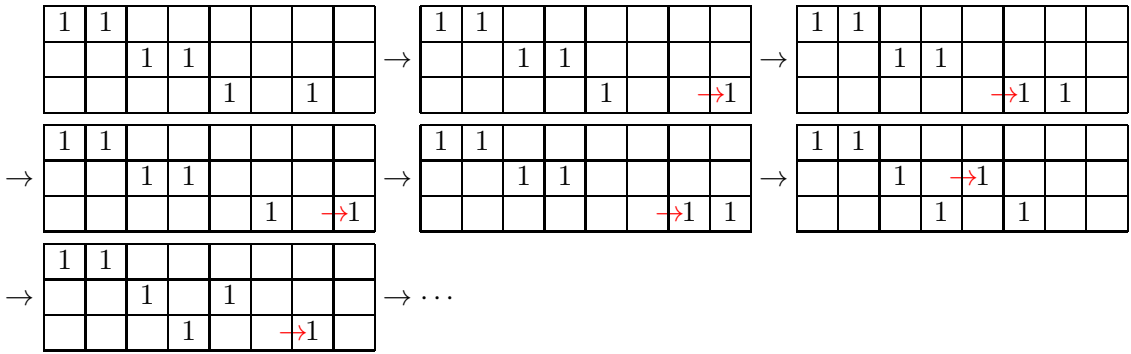
\begin{align*}
&\begin{array}{|c|c|c|c|c|c|c|c|}
\hline
1 & 1&&&&&&\\\hline
&& 1 & 1&&&&\\\hline
&&&& 1 & \phantom{0} & 1 & \phantom{0}\\\hline
\end{array} \to 
\begin{array}{|c|c|c|c|c|c|c|c|}
\hline
1 & 1&&&&&&\\\hline
&& 1 & 1&&&&\\\hline
&&&& 1 & \phantom{0} & \phantom{0}\makebox[0pt][l]{\color{red}$\rightarrow$} & 1\\\hline
\end{array}
\to
\begin{array}{|c|c|c|c|c|c|c|c|}
\hline
1 & 1&&&&&&\\\hline
&& 1 & 1&&&&\\\hline
&&&& \phantom{0}\makebox[0pt][l]{\color{red}$\rightarrow$} &1 & 1 &\phantom{0}\\\hline
\end{array}\\
\to~&\begin{array}{|c|c|c|c|c|c|c|c|}
\hline
1 & 1&&&&&&\\\hline
&& 1 & 1&&&&\\\hline
&&&& \phantom{0} &1 &\phantom{0}\makebox[0pt][l]{\color{red}$\rightarrow$} & 1\\\hline
\end{array}\to
\begin{array}{|c|c|c|c|c|c|c|c|}
\hline
1 & 1&&&&&&\\\hline
&& 1 & 1&&&&\\\hline
&&&& \phantom{0}  &\phantom{0}\makebox[0pt][l]{\color{red}$\rightarrow$} &1& 1\\\hline
\end{array} \to
\begin{array}{|c|c|c|c|c|c|c|c|}
\hline
1 & 1&&&&&&\\\hline
&& 1 & \phantom{0}\makebox[0pt][l]{\color{red}$\rightarrow$}&1&&&\\\hline
&&&1&  &1 &\phantom{0} & \phantom{0}\\\hline
\end{array}\\
\to~&\begin{array}{|c|c|c|c|c|c|c|c|}
\hline
1 & 1&&&&&&\\\hline
&& 1 & &1&&&\\\hline
&&&1&  &\phantom{0}\makebox[0pt][l]{\color{red}$\rightarrow$} &1& \phantom{0}\\\hline
\end{array} \to \cdots
\end{align*}
\end{ex}

In contrast to the example above, the next subsection takes care of $Q(S)$ and $Q(T)$ in Eq.~\eqref{eq:AlgoCond}.

\subsection{The \enquote{marking the quotient} routine}
To take care of the quotient sets in Eq.~\eqref{eq:AlgoCond} we mark the quotient of each row in $\bs C$ in the row itself. This ensures that rows below this row don't use elements of the quotient sets.

\begin{ex}
We use the same example as above. We start with $b_S^* = (1,1,0,0,0,0,0,0)$, which means that
\[
\bs C = 
\begin{array}{|c|c|c|c|c|c|c|c|}
\hline
1 & 1&&&&&&\\\hline
&& \phantom{0} & \phantom{0}&&&&\\\hline
&&&& \phantom{0} & \phantom{0} & \phantom{0} & \phantom{0}\\\hline
\end{array}.\]
We mark the quotient set $Q(S)$ in line $b_S^*$ with a \enquote{$q$}:
\[
\bs C = 
\begin{array}{|c|c|c|c|c|c|c|c|}
\hline
1 & 1&&q&&&&\\\hline
&& \phantom{0} & \phantom{0}&&&&\\\hline
&&&& \phantom{0} & \phantom{0} & \phantom{0} & \phantom{0}\\\hline
\end{array}.\]
So the first possible $b_T^*$ line is
\[
\bs C = 
\begin{array}{|c|c|c|c|c|c|c|c|}
\hline
1 & 1&&q&&&&\\\hline
&& 1 & &1 &&&\\\hline
&&&& \phantom{0} & \phantom{0} & \phantom{0} & \phantom{0}\\\hline
\end{array}.\]
\end{ex}

Note that $X\subseteq Q(X)$ for all $X\in \{S,T,U\}$. Thus, we only have to mark the elements in $Q(X)\setminus X=:\bar Q(X)$. Before we can move a 1 in a row $b_X^*$ we have to delete all marks $\bar Q(X)$.

We have to deal with the case, that we found a $b_T^*$ with the \enquote{moving 1} principle, but $Q(S) \cap Q(T) \neq \{1\}$: In this situation we have to undo all steps in the process of \enquote{marking all elements in $\bar Q(T)$} and we have to find a new $b_T^*$ by moving a 1.

\subsection{Efficient Storage of the Basic Binary Representation Matrix}
If we use the matrix $\bs C$ to store all necessary information we have to store $3n$ elements and we need exactly 3 tests to check if we can move a 1 to a position $p$: we have to check if $(b_S^*)_p=(b_T^*)_p=(b_U^*)_p=0$.

We can omit the unnecessary space of $2n$ elements and the unnecessary 2 tests by projecting $\bs C^{3\times n}$ to a vector $\mi{marked}\in\{-2,-1,0,1,2,3\}^n$:
\[
\bs C
\quad \mapsto \quad
1\cdot b_S^* + (-1) \cdot b_{\bar Q(S)}^*+2\cdot b_T^* + (-2) \cdot b_{\bar Q(T)}^*
+3\cdot b_U^*
\]

\begin{ex}
\newcommand{\pn}{\phantom{0}}
Consider the basic binary representation matrix
\begin{align*}
\bs C & = \phantom{(}
\begin{array}{|c|c|c|c|c|c|c|c|c|c|c|c|c|}
\hline
1 & 1 &   & q &   & q &   &   &   &   &   &   &\\\hline
  &   & 1 &   & 1 &   & q &   & q & q &   &   &\\\hline
  &   &   &   &   &   &   & 1 &   &   & 1 &\pn&\pn\\\hline
\end{array}.\\
\intertext{The corresponding $\mi{marked}$ vector is}
\mi{marked} &=(
\begin{array}{ccccccccccccc}
1\rlap{,}&1\rlap{,}&2\rlap{,}&\llap{-}1\rlap{,}&2\rlap{,}&\llap{-}1\rlap{,}&\llap{-}2\rlap{,}&3\rlap{,}&\llap{-}2\rlap{,}&\llap{-}2\rlap{,}&3\rlap{,}&0\rlap{,}&0
\end{array})
\end{align*}
\end{ex}

The check $(b_S^*)_p=(b_T^*)_p=(b_U^*)_p=0$ can now be done with $\mi{marked}[p]=0$.

\subsection{The Search Algorithm}
The listing \enquote{SearchTPPTripleOfGivenType($G$, $m$)} shows the pseudo-code for the main function of the search algorithm. The interested reader can get a more detailed version of this pseudo-code, all other pseudo-codes and an implementation in GAP\nocite{GAP4.6.3} online \cite{online:FiveFiveFive} or via e-mail from the second author. 

\begin{algorithm}
\NoCaptionOfAlgo
\caption{SearchTPPTripleOfGivenType($G$, $m$)}
\For(\tcp*[f]{start with $S=\{1_G,g_1,g_2,\ldots,g_{m-1}\}$}){$i=1,\ldots,m-1$}{$\mi{marked}[i]:=1$\;}
\Repeat{\textup{it is not possible to use the \enquote{moving 1} principle for $b_S^*$ anymore}}{
mark quotient set $\bar Q(S)$ of row $b_S^*$\;
\If{\textup{it is possible to generate a feasible row $b_T^*$ from scratch}}{
\Repeat{\textup{it is not possible to use the \enquote{moving 1} principle for $b_T^*$ anymore}}{
\If{\textup{it is possible to mark the quotient set $\bar Q(T)$ of $b_T^*$ without a conflict with $Q(S)$}}{
\If{\textup{it is possible to generate a feasible row $b_U^*$ from scratch}}{
\Repeat{\textup{it is not possible to use the \enquote{moving 1} principle for $b_U^*$ anymore}}{
\If(\tcp*[f]{use a test from \cite{HedtkeMurthy2011}}){\textup{$(S,T,U)$ is a TPP triple}}{\Return{$(S,T,U)$}\;}
}
}
unmark the quotient set $\bar Q(T)$ of $b_T^*$\;
}
}
}
unmark the quotient set $\bar Q(S)$ of $b_S^*$\;
}
\end{algorithm}

To test if a given candidate satisfies the TPP, we can use the test algorithms from Hedtke and Murthy \cite{HedtkeMurthy2011}. It would also be possible to use a specialized TPP test, because $Q(S)$ and $Q(T)$ are already known and they satisfy Eq.~\eqref{eq:AlgoCond}.

\section{An Application for $5\times 5$ Matrix Multiplication}
\noindent In this section, we describe an application of the new algorithm. We
will show that if a finite group $G$ admits a \tens{5,5,5} triple, then
$R(G) \geq 100$. That is, we cannot improve the current best bound
for $R(5)$ using this particular TPP approach --  of course there
may be other group-theoretic methods that do yield better bounds.
Even with the new algorithm, it is not feasible simply to test all
groups of order less than 100 for \tens{5,5,5} triples. Therefore we must
use theoretical methods to reduce the list of candidates that must
be checked. We will also produce a list of groups that could in
theory contain a \tens{5,5,5} triple for which $\underline R(G) < 125$ (as defined below).

For a finite group $G$, let $T(G)$ be the number of irreducible
complex characters of $G$
 and $b(G)$ the largest degree of an irreducible character of $G$.

We start with two known results.

\begin{thm}\textup{\cite[Theorem 6 and Remark 2]{Pospelov2011}}\label{popsthm}
Let $G$ be a group.
\begin{enumerate}
\item If $b(G) = 1$, then $R(G) = |G|$.
\item If $b(G) = 2$, then $R(G) = 2|G| - T(G)$.
\item If $b(G) \geq 3$, then $R(G) \geq 2|G| + b(G) - T(G) - 1$.
\end{enumerate}
\end{thm}

We write $\overline{R}(G):=\sum_i R(d_i)$ for the best known upper bound (follows from Wedderburn's structure theorem) and $\underline{R}(G)$ for the best known lower bound (the theorem above) for $R(G)$.

\begin{defn}[C1 and C2 candidates]
A group $G$ that realizes $\tens{5,5,5}$ and satisfies $\uR(G)<100$
will be called \emph{C1 candidate}. A group $G$ that realizes
$\tens{5,5,5}$ and satisfies $\uR(G)<125$ will be called \emph{C2
candidate}.
\end{defn}

The following is well known, but we include a short proof for ease
of reference.
\begin{lem}
If $G$ is non-abelian, then $T(G) \leq \frac{5}{8}|G|$. Equality
implies that $|G: Z(G)| = 4$.
\end{lem}

\begin{proof}
If the quotient $G/Z(G)$ is cyclic, then $G$ is abelian. Therefore
if $G$ is non-abelian, then $|G:Z(G)| \geq 4$. Hence $|Z(G)| \leq
\frac{1}{4}|G|$. Now $T(G)$ is known to equal the number of conjugacy classes of $G$. For
any $x \in G$, either $x$ is central or $|x^G| \geq 2$. The number
of conjugacy classes of length at least 2 is $T(G) - |Z(G)|$.
Therefore $|G| \geq |Z(G)| + 2(T(G) - |Z(G)|)$. This implies $T(G)
\leq \frac{1}{2}(|G| + |Z(G)|) \leq \frac{5}{8}|G|$. Equality is
only possible when $|Z(G)| = \frac{1}{4}|G|$.
\end{proof}

Obviously, it is necessary to keep the list of all C1 and C2
candidates as short as possible. To achieve this goal we will
develop some common properties of C1 and C2 candidates in this
section. We will use them to eliminate as many candidates as
possible from the list.

It will be helpful to establish some notation in the particular case
where a group has a TPP triple and a subgroup of index $2$.

\begin{defn}\label{def:capH}
Let $G$ be a group with a TPP triple $(S,T,U)$, and suppose $H$ is a
subgroup of index $2$ in $G$. We define $S_0 = S \cap H$, $T_0 =
T\cap H$, $U_0 = U \cap H$, $S_1 = S \setminus H$, $T_1 = T
\setminus H$ and $U_1 = U\setminus H$.
\end{defn}

\begin{lem}\label{subgp}
Suppose $G$ realizes $\tens{5,5,5}$. If $G$ has a subgroup $H$ of
index $2$, then $H$ realizes $\tens{3,3,3}$.
\end{lem}

\begin{proof}
Suppose $G$ realizes $\tens{5,5,5}$ via the TPP triple $(S, T, U)$.
If $|S_0| < |S_1|$, then for any $a \in S_1$, replace $S$ by
$Sa^{-1}$. This will have the effect of interchanging $S_0$ and
$S_1$. Hence we may assume that $|S_0| \geq |S_1|$, $|T_0| \geq
|T_1|$ and $|U_0| \geq |U_1|$. Now $(S_0, T_0, U_0)$ is a TPP triple
of $H$, and since each of $S_0$, $T_0$ and $U_0$ has at least $3$
elements, clearly $H$ realizes $\tens{3,3,3}$.
\end{proof}

\begin{lem}\label{tech}
Suppose $G$ has a TPP triple $(S,T,U)$. Let $H$ be an abelian
subgroup of index 2 in $G$. Then the following hold.
\begin{enumerate}
\item[a)] $|S_0^{-1}T_0U_0| = |S_0||T_0||U_0|;$
\item[b)] $|S_1^{-1}T_1U_0| \geq  |S_1||T_1|$;
\item[c)] $|S_1^{-1}U_1|  = |S_1||U_1|$;
\item[d)] $S_0^{-1}T_0U_0 \cap S_1^{-1}T_1U_0 = \emptyset$;
\item[e)] $S_0^{-1}T_0U_0 \cap S_1^{-1}U_1T_0 = \emptyset$;
\item[f)] $S_1^{-1}T_1U_0 \cap S_1^{-1}U_1T_0 = \emptyset$.
\end{enumerate}
\end{lem}

\begin{proof} The proof relies almost entirely on the definition of
a TPP triple $(S,T,U)$; that if $s \in Q(S), t \in Q(T)$ and $u \in
Q(U)$ with $stu = 1$, then $s = t = u = 1$.
\begin{enumerate}
\item[a)] The map $(s,t,u) \mapsto s^{-1}tu$ from
$S_0 \times T_0 \times U_0$ to $S_0^{-1}T_0U_0$ is clearly
surjective. It is also injective: suppose $s^{-1}tu = \hat
s^{-1}\hat t \hat u$ for some $s,\hat s \in S_0$, $t,\hat t \in T_0$ and $u,\hat u \in U_0$. Then, remembering that $H$ is abelian, we may
rearrange to get $(\hat s s^{-1})(t \hat t^{-1})(u \hat u^{-1}) =
1$, forcing (by definition of TPP triple), $s = \hat s$, $t = \hat
t$, $u = \hat u$. Therefore the map is bijective and
$|S_0^{-1}T_0U_0| = |S_0||T_0||U_0|$.

\item[b)] The map $(s_1, t_1) \mapsto s_1^{-1}t_11$ from $S_1 \times
T_1$ to $S_1^{-1}T_1U_0$ is injective as $s_1^{-1}t_11 = \hat
s_1^{-1}\hat t_11$, for some $s_1,\hat s_1 \in S_1$ and $t_1,\hat t_1\in T_1$, implies $(\hat s_1 s_1^{-1})(t_1\hat
t_1^{-1})(11^{-1}) = 1$, which implies $s_1 = \hat s_1$ and $t_1 =
\hat t_1$. Thus $|S_1^{-1}T_1U_0| \geq |S_1||T_1|$.%

\item[c)] The map $(s_1, u_1) \mapsto s_1^{-1}u_1$ from $S_1 \times U_1$
to $S_1^{-1}U_1$ is clearly surjective; it is injective as
$s_1^{-1}u_1 = \hat s_1^{-1} \hat u_1$ implies $(\hat s_1s_1^{-1})(1
1^{-1})(u_1\hat u_1^{-1}) = 1$ and hence $s_1 = \hat s_1$ and $u_1 =
\hat u_1$. Therefore
$|S_1^{-1}U_1|  = |S_1||U_1|$.%

\item[d)] A nonempty intersection $S_0^{-1}T_0U_0 \cap S_1^{-1}T_1U_0 \neq \emptyset$
 implies there exist $s_0 \in S_0$, $t_0 \in T_0$, $u_0, \hat u_0 \in U_0$,
 $s_1 \in S_1$ and $t_1 \in T_1$ such that $s_0^{-1}t_0u_0 =
 s_1^{-1}t_1\hat u_0$. But then $t_1^{-1}s_1 s_0^{-1} t_0u_0\hat
 u_0^{-1} = 1$. Now $t_1^{-1}s_1$, $s_0$  and $t_0$ are all elements of the abelian group $H$.
 Therefore we can rearrange to get
$(t_0t_1^{-1})(s_1s_0^{-1})(u_0\hat
 u_0^{-1}) = 1$.
Since $(T, S, U)$ is a TPP triple, this implies $s_0 = s_1$,
contradicting the fact that $s_0$ and $s_1$ lie in different
$H$-cosets. Therefore
$S_0^{-1}T_0U_0 \cap S_1^{-1}T_1U_0 = \emptyset$.%

\item[e)] Suppose for some $s_0 \in S_0$, $t_0, \hat t_0 \in T_0$, $u_0 \in U_0$, $s_1
\in S_1$ and $u_1 \in U_1$ we have $s_0^{-1}t_0u_0 = s_1^{-1}u_1\hat
t_0$. Then $(s_0s_1^{-1})(u_1u_0^{-1})(\hat t_0t_0^{-1}) = 1$, which
implies (by the TPP for $(S,U,T)$) that $s_0 = s_1$, a
contradiction. Therefore
$S_0^{-1}T_0U_0 \cap S_1^{-1}U_1T_0 = \emptyset$.%

\item[f)] Suppose for some $s_1, \hat s_1 \in S_1$, $t_0 \in T_0$, $t_1 \in T_1$, $u_0 \in U_0$
and $u_1\in U_1$, we have $s_1^{-1}t_1u_0 = \hat s_1^{-1}u_1t_0$.
Then $(\hat s_1s_1^{-1})(t_1t_0^{-1})(u_0u_1^{-1}) = 1$, which
implies $u_0 = u_1$, a contradiction. Therefore $S_1^{-1}T_1U_0 \cap
S_1^{-1}U_1T_0 = \emptyset$.\qedhere
\end{enumerate}
\end{proof}

\begin{thm} \label{leq72}
If $G$ realizes $\tens{5,5,5}$ and $|G| \leq 72$, then $G$ has no
abelian subgroups of index $2$.
\end{thm}

\begin{proof}
Suppose $G$ has an abelian subgroup $H$ of index 2 and realizes
$\tens{5,5,5}$ via the TPP triple $(S, T, U)$. Define $S_0$, $T_0$,
$U_0$, $S_1$, $T_1$ and $U_1$ as in Definition~\ref{def:capH}. Then, as in
the proof of Lemma~\ref{subgp}, we may assume $|S_0| \geq 3$, $|T_0| \geq
3$ and $|U_0| \geq 3$. Without loss of generality we may assume that
$|S_0| \geq |T_0|$ and $|S_0| \geq |U_0|$. Now since $|G|
\leq 72$, we have $|H| \leq 36$. So, from Lemma~\ref{tech} we have%
\begin{eqnarray}
36 \geq |H| &\geq& |S_0^{-1}T_0U_0 \cup S_1^{-1}U_1T_0 \cup S_1^{-1}T_1U_0| \nonumber\\
               &=& |S_0||T_0||U_0| + |S_1^{-1}U_1T_0| + |S_1^{-1}T_1U_0|\label{eq:1}\\
            &\geq& |S_0||T_0||U_0| + |S_1||U_1| + |S_1||T_1|\label{eq:2}.
\end{eqnarray}
Using Equation~\eqref{eq:2} if either $T_0 \geq 4$ or $U_0 \geq 4$, we have
$S_0 \geq 4$, which forces $|H| \geq 48$, a contradiction. Thus
$|T_0| = |U_0| = 3$. If $S_0 \geq 4$ then we get $|H| \geq 40$,
another contradiction. Therefore $|S_0| = |T_0| = |U_0| = 3$, which
gives that $|H| \geq 27 + 4 + 4 = 35$, and so $|H| \in \{35, 36\}$.
If two of $Q(S_0)$, $Q(T_0)$ and $Q(U_0)$ were groups of order 4,
then they would generate a subgroup of order 16 in $H$, which is
impossible. Therefore, permuting $S, T$ and $U$ if necessary, we may
assume that $Q(T_0)$ and $Q(U_0)$ are not subgroups of order 4.

Now consider $S_1^{-1}U_1T_0$. Write $X = S_1^{-1}U_1$. Then $|X| =
4$. If $|XT_0| = 4$, then $XT_0 = X$, and thus $X\langle T_0\rangle=
X$, which implies that $X$ is a union of $\langle
T_0\rangle$-cosets. In particular, 4 = $|X|$ divides the order of
$\langle T_0\rangle$. But $T_0$ alone contains 3 elements. Hence
$\langle T_0\rangle$ has order~4. A quick check shows that $Q(T_0) =
\langle T_0\rangle$, contradicting the fact that $Q(T_0)$ is not a
subgroup of order~4. We have therefore shown that $|S_1^{-1}U_1T_0|
> 4$. A similar argument with $S_1^{-1}T_1U_0$ and $Q(U_0)$ shows
that $|S_1^{-1}T_1U_0| > 4$. Substituting back into Equation~\eqref{eq:1}
gives $|H| \geq 27 + 5 + 5 = 37$, a contradiction. Therefore no
group of order at most 72 can have both a $\tens{5,5,5}$ triple and
an abelian subgroup of index~2.
\end{proof}

We are grateful to Peter M. Neumann for pointing out an argument
which considerably shortened our proof for the case $|H| = 36$ in
the above result.

\subsection{C1 Candidates}

\begin{prop} \label{upplow}
If $G$ is a C1 candidate, then $G$ is non-abelian and $45 \leq |G|
\leq 72$.
\end{prop}

\begin{proof}
If $G$ is abelian then $R(G) = |G|$. The maximal size of a TPP
triple that $G$ can realize is $|G|$. Therefore $G$ cannot be a C1
candidate. Assume then that $G$ is non-abelian. The fact that $|G|
\geq 45$ follows immediately from Lemma~\ref{neumannthm}. For the
upper bounds, the fact that $T(G) \leq \frac{5}{8}|G|$ implies $2|G|
- T(G) \geq \frac{11}{8}|G|$ and hence, by Theorem \ref{popsthm},
$R(G) \geq \frac{11}{8}|G|$. So if $|G| > 72$, then $R(G) > \frac{11
\times 72}{8} = 99$. Hence $G$ is not a C1 candidate. Therefore, if
$G$ is a C1 candidate, then $45 \leq |G| \leq 72$.
\end{proof}

\begin{thm} \label{64}
No group of order $64$ is a C1 candidate.
\end{thm}

\begin{proof}
A GAP calculation of Pospelov's lower bound on $R(G)$, followed by
elimination of any group with an abelian subgroup of index $2$,
leaves a possible list of seven groups of order $64$ that could be
C1 candidates. If any of these groups $G$ were to realize a
$\tens{5,5,5}$ triple, then any subgroup of order $32$ in $G$ would
realize a $\tens{3,3,3}$ triple. But a brute-force computer search,
similar to that performed by two of the current authors in
\cite{HedtkeMurthy2011}, shows that each of these groups of order
$64$ has at least one subgroup of order $32$ which does not realize
$\tens{3,3,3}$. Therefore, no group of order 64 is a C1
candidate.\end{proof}

\begin{table}
\centering
\begin{tabular}{lllrr}
\toprule
GAP ID & structure & character degree pattern & $\uR(G)$ & $\oR(G)$\\
\midrule
{}[48,3]  & $C_4^2 \rtimes C_3$                  & $(1^3,3^5)$         & 90 & 118\\
{}[48,28] & $C_2 . S_4 = \mathrm{SL}(2,3) . C_2$ & $(1^2,2^3,3^2,4^1)$ & 91 & 118\\
{}[48,29] & $\mathrm{GL}(2,3)$                   & $(1^2,2^3,3^2,4^1)$ & 91 & 118\\
{}[48,30] & $A_4 \rtimes C_4$                    & $(1^4,2^2,3^4)$     & 88 & 110\\
{}[48,31] & $C_4 \times A_4$                     & $(1^{12},3^4)$      & 82 & 104\\
{}[48,32] & $C_2 \times \mathrm{SL}(2,3)$        & $(1^6,2^6,3^2)$     & 84 &  94\\
{}[48,33] & $\mathrm{SL}(2,3) \rtimes C_2$       & $(1^6,2^6,3^2)$     & 84 &  94\\
{}[48,48] & $C_2 \times S_4$                     & $(1^4,2^2,3^4)$     & 88 & 110\\
{}[48,49] & $C_2^2 \times A_4$                   & $(1^{12},3^4)$      & 82 & 104\\
{}[48,50] & $C_2^4 \rtimes C_3$                  & $(1^3,3^5)$         & 90 & 118\\
{}[54,10] & $C_2 \times (C_3^2 \rtimes C_3)$     & $(1^{18},3^4)$      & 88 & 110\\
{}[54,11] & $C_2 \times (C_9 \rtimes C_3)$       & $(1^{18},3^4)$      & 88 & 110\\
\bottomrule
\end{tabular}
\caption{All possible C1 candidates.} \label{tab:c1cand}
\end{table}

\begin{thm}
Table~\ref{tab:c1cand} contains all possible C1 candidates.
\end{thm}

\begin{proof}
By Proposition~\ref{upplow} we need only look at groups of order
between $45$ and $72$. A simple GAP program can calculate Pospelov's
lower bound on $R(G)$. Any group for which this bound is greater
than $99$ can be eliminated. Next, we can eliminate any group with
an abelian subgroup of index~$2$ by Theorem~\ref{leq72}, and any
group of order~$64$ by Theorem~\ref{64}. This reduces the list to
$20$ groups. Finally, we observe that if any group of order 48 is a
candidate, then any of its subgroups of order~$24$ must realize a
$\tens{3,3,3}$ triple. Another brute-force search on groups of order
24 eliminates ten groups of order~$48$ from the list. The final list
contains ten groups of order~$48$ and two of order~$54$.
\end{proof}

\subsection{C2 Candidates}

\begin{prop} \label{upplowC2}
If $G$ is a C2 candidate, then $G$ is non-abelian and $45 \leq |G|
\leq 90$.
\end{prop}

\begin{proof}
We use the same arguments as in the proof of
Proposition~\ref{upplow}: If $|G| \geq 91$, then $R(G) \geq \frac{11
\times 91}{8}
> 125$. Hence $G$ is not a C2 candidate. Therefore if $G$ is a C2
candidate, then $45 \leq |G| \leq 90$.
\end{proof}

\begin{table}
\centering
\begin{tabular}{lllrr}
\toprule
GAP ID & structure & character degree pattern & $\uR(G)$ & $\oR(G)$\\
\midrule
{}[52,3]  & $C_{13} \rtimes C_4$                        & $(1^4,4^3)$       & 100 & 151\\
{}[54,5]  & $(C_3^2 \rtimes C_3) \rtimes C_2$           & $(1^6,2^3,6^1)$   & 103 & 188\\
{}[54,6]  & $(C_9 \rtimes C_3) \rtimes C_2$             & $(1^6,2^3,6^1)$   & 103 & 188\\
{}[54,8]  & $(C_3^2 \rtimes C_3) \rtimes C_2$           & $(1^2,2^4,3^4)$   & 100 & 122\\
{}[55,1]  & $C_{11} \rtimes C_5$                        & $(1^5,5^2)$       & 107 & 205\\
{}[56,11] & $C_2^3 \rtimes C_7$                         & $(1^7,7^1)$       & 110 & 265\\
{}[57,1]  & $C_{19} \rtimes C_3$                        & $(1^3,3^6)$       & 107 & 141\\
{}[60,5]  & $A_5$                                       & $(1^1,3^2,4^1,5^1)$&119 & 196\\
{}[60,6]  & $C_3 \times (C_5 \rtimes C_4)$              & $(1^{12},4^3)$    & 108 & 159\\
{}[60,7]  & $C_{15} \rtimes C_4$                        & $(1^4,2^2,4^3)$   & 114 & 165\\
{}[60,8]  & $S_3 \times D_{10}$                         & $(1^4,2^6,4^2)$   & 111 & 144\\
{}[60,9]  & $C_5 \times A_4$                            & $(1^{15},3^5)$    & 102 & 130\\
{}[63,1]  & $C_7 \rtimes C_9$                           & $(1^9,3^6)$       & 113 & 147\\
{}[63,3]  & $C_3 \times (C_7 \rtimes C_3)$              & $(1^9,3^6)$       & 113 & 147\\
{}[72,16] & $C_2 \times (C_2^2 \rtimes C_9)$            & $(1^{18},3^6)$    & 122 & 156\\
{}[72,47] & $C_6 \times A_4$                            & $(1^{18},3^6)$    & 122 & 156\\
{}[78,3]  & $C_{13} \times S_3$                         & $(1^{26},2^{13})$ & 117 & 117\\
{}[80,21] & $C_5 \times ((C_4 \times C_2) \rtimes C_2)$ & $(1^{40},2^{10})$ & 110 & 110\\
{}[80,22] & $C_5 \times (C_4 \rtimes C_4)$              & $(1^{40},2^{10})$ & 110 & 110\\
{}[80,24] & $C_5 \times (C_8 \rtimes C_2)$              & $(1^{40},2^{10})$ & 110 & 110\\
{}[80,46] & $C_{10} \times D_8$                         & $(1^{40},2^{10})$ & 110 & 110\\
{}[80,47] & $C_{10} \times Q_8$                         & $(1^{40},2^{10})$ & 110 & 110\\
{}[80,48] & $C_5 \times ((C_4 \times C_2) \rtimes C_2)$ & $(1^{40},2^{10})$ & 110 & 110\\
{}[88,9]  & $C_{11} \times D_8$                         & $(1^{44},2^{11})$ & 121 & 121\\
{}[88,10] & $C_{11} \times Q_8$                         & $(1^{44},2^{11})$ & 121 & 121\\
\bottomrule
\end{tabular}
\caption{All possible C2 candidates that are not C1 candidates.}
\label{tab:c2cand}
\end{table}

\begin{thm}
Table~\ref{tab:c2cand} contains all possible C2 candidates that are
not C1 candidates.
\end{thm}

\begin{proof}
By Proposition~\ref{upplowC2}, we can restrict our attention to
groups of order between $45$ and $90$. We can use Pospelov's bound
for $R(G)$ and (for groups of order at most 72) the existence of
abelian subgroups of index~$2$ to eliminate many candidates. After
these observations, we look to see if any of the remaining
candidates have subgroups of index~$2$ that do not realize
$\tens{3,3,3}$. If so, then by Lemma~\ref{subgp}, the group cannot
be a C2 candidate. After this process, $37$ groups remain as
candidates. Twelve are the existing C1 candidates we already know
about. So there are $25$ `new' groups here.
\end{proof}

We note that one of the C2 candidates, $A_5$, is already known
(\cite[Section~3]{Neumann2011}) to have a $\tens{5,5,5}$ triple so
we would not need to check it again computationally.

\section{Computations, Tests and Results}

\subsection{Runtime}
We tested our new search algorithm against a specialized
version (that only looks for \tens{m,m,m} triples) of the currently best known search algorithm with the test routine \texttt{TPPTestMurthy} (see \cite{HedtkeMurthy2011}). Note that we only consider groups that do not realize \tens{3,3,3} to show the worst-case runtimes of the searches. Table~\ref{tab:runtime} lists the runtimes\footnote{The test were made with GAP 4.6.3 64-bit (compiled with GCC 4.2.1 on OS X 10.8.3 using the included Makefile) on a Intel$^\text{\textregistered}$ Core\texttrademark{} i7-2820QM CPU @ 2.30GHz machine with 8 GB DDR3 RAM @ 1333MHz.
} of the search for \tens{3,3,3} TPP triples in non-abelian groups of order up to 26 that satisfy Neumann's inequality $3(3+3-1)\leq |G|$. 
\begin{table}
\small
\begin{tabular}{llrrrrrr}
\toprule
GAP& & \multicolumn{2}{c}{average$^*$ runtime in ms} &speed-& \multicolumn{2}{c}{number of TPP tests} & search space re-\\
ID & structure & new algo.\ & old algo.\ &up& new algo.\ & old algo.&duction factor$^{**}$\ \\
\midrule
{}[16,3]  & $(C_4 \times C_2) \rtimes C_2$ &      192 &  20{,}133 & 104 &      11{,}595 &     450{,}450 & 38\\
{}[16,4]  & $C_4 \rtimes C_4$              &      140 &  19{,}481 & 139 &             0 &     450{,}450 & $\infty$\\
{}[16,6]  & $C_8 \rtimes C_2$              &      116 &  19{,}631 & 169 &             0 &     450{,}450 & $\infty$\\
{}[16,7]  & $D_{16}$                       &      241 &  20{,}416 &  84 &      14{,}336 &     450{,}450 & 31\\
{}[16,8]  & $\mathit{QD}_{16}$             &      162 &  20{,}060 & 123 &       9{,}005 &     450{,}450 & 50\\
{}[16,9]  & $Q_{16}$                       &       99 &  19{,}250 & 194 &             0 &     450{,}450 & $\infty$\\
{}[16,11] & $C_2 \times D_8$               &      311 &  20{,}079 &  64 &      19{,}314 &     450{,}450 & 23\\
{}[16,12] & $C_2 \times Q_8$               &      135 &  18{,}667 & 138 &       7{,}628 &     450{,}450 & 59\\
{}[16,13] & $(C_4 \times C_2)\rtimes C_2$  &      201 &  19{,}538 &  97 &      12{,}107 &     450{,}450 & 37\\
\midrule
{}[18,1]  & $D_{18}$                       &      658 &  51{,}899 &  78 &      39{,}499 & 1{,}113{,}840 & 28\\
{}[18,3]  & $C_3 \times S_3$               &      341 &  50{,}360 & 147 &      20{,}134 & 1{,}113{,}840 & 55\\
{}[18,4]  & $C_3^2\rtimes C_2$             &      646 &  51{,}131 &  79 &      39{,}999 & 1{,}113{,}840 & 27\\
\midrule
{}[20,1]  & $C_5 \rtimes C_4$              &  1{,}028 & 119{,}588 & 116 &      54{,}233 & 2{,}441{,}880 & 45\\
{}[20,3]  & $C_5 \rtimes C_4$              &  1{,}388 & 121{,}702 &  87 &      73{,}971 & 2{,}441{,}880 & 33\\
{}[20,4]  & $D_{20}$                       &  2{,}033 & 118{,}599 &  58 &     114{,}979 & 2{,}441{,}880 & 21\\
\midrule
{}[22,1]  & $D_{22}$                       &  4{,}539 & 241{,}524 &  53 &     248{,}950 & 4{,}883{,}760 & 19\\
\midrule
{}[24,1]  & $C_3 \rtimes C_8$              &  5{,}610 & 501{,}854 &  89 &     292{,}340 & 9{,}085{,}230 & 31\\
{}[24,4]  & $C_3 \rtimes  Q_8$             &  6{,}056 & 498{,}571 &  82 &     303{,}162 & 9{,}085{,}230 & 29\\
{}[24,5]  & $C_4 \times  S_3$              &  7{,}912 & 483{,}640 &  61 &     419{,}556 & 9{,}085{,}230 & 21\\
{}[24,6]  & $D_{24}$                       & 10{,}711 & 479{,}688 &  44 &     568{,}672 & 9{,}085{,}230 & 15\\
{}[24,7]  & $C_2 \times (C_3 \rtimes C_4)$ &  6{,}623 & 486{,}323 &  73 &     339{,}829 & 9{,}085{,}230 & 26\\
{}[24,8]  & $(C_6 \times C_2) \rtimes C_2$ &  8{,}804 & 479{,}182 &  54 &     463{,}453 & 9{,}085{,}230 & 19\\
{}[24,10] & $C_3 \times D_8$               &  6{,}540 & 481{,}217 &  73 &     359{,}830 & 9{,}085{,}230 & 25\\
{}[24,11] & $C_3 \times Q_8$               &  5{,}250 & 490{,}716 &  93 &     284{,}001 & 9{,}085{,}230 & 31\\
{}[24,14] & $C_2 \times C_2 \times S_3$    & 11{,}555 & 475{,}916 &  41 &     622{,}455 & 9{,}085{,}230 & 14\\
\midrule
{}[26,1]  & $D_{26}$                       & 20{,}658 & 832{,}722 &  40 & 1{,}024{,}317 & 15{,}939{,}000& 15\\
\midrule
\multicolumn{8}{l}{$^*$ The average is taken over 10 runs in which the highest and lowest runtimes are omitted.}\\
\multicolumn{8}{l}{$^{**}$ A factor $X$ means that $(\text{\# TPP test of the new algo.})\leq\frac1X(\text{\# TPP test of the old algo.})$.}\\
\bottomrule
\end{tabular}
\caption{Comparison of the average runtime and number of TPP tests in the search of \tens{3,3,3} TPP triples for the old and the new search algorithm.}
\label{tab:runtime}
\end{table}
Our algorithm achieves a speed-up of 40 in the worst-case and 194 in the best-case in comparison to the specialized version of Hedtke and Murthy \cite{HedtkeMurthy2011}. We are able to shrink the number of candidates that we have to test for the TPP by a factor of 14 in the worst-case and 59 in the best-case. We remark that there are cases where the old algorithm tests 450{,}450 candidates and the new algorithm requires no TPP tests at all.

We only did tests in the \tens{3,3,3} case, because the old algorithm is too slow to do a comparison like Table~\ref{tab:runtime} for the \tens{4,4,4} case (or higher). 
The search need only be run in groups that satisfy Neumann's inequality: a group $G$ can only realize \tens{m,m,m} if it satisfies $m(2m-1)\leq |G|$.

We remark that the speed-up becomes slower when the group becomes larger. However this is not of particular concern in the context of our problem: the old search algorithm works on $S$, $T$ and $U$ and the new algorithm works on $Q(S)$, $Q(T)$ and $U$. So in the best-case the old algorithm uses $|S|+|T|+|U|=3m$ elements and the new algorithm uses $|Q(S)|+|Q(T)|+|U|\leq m^2+m^2+m$ elements to filter TPP triple candidates. The speed-up will be problematically small when $m^2 \ll |G|$, but you will only look for groups that are near Neumann's lower bound to get a good matrix multiplication algorithm.

It is not easy to get results about the asymptotic runtime, because that highly depends on the structure of the groups. But as a worst-case result we get
\[
\underbrace{\mathcal O \left( \frac{|G|!}{m!^3(|G|-3m)!} \right)}_{\substack{\text{bound for the number of}\\\text{triples that satisfy Eq.~\eqref{eq:AlgoCond}}}}
~\times
\underbrace{\mathcal O \left(m^4\log m \right)}_{\substack{\text{worst-case runtime}\\\text{for a TPP test with}\\\text{Thm.~\ref{thm:HedtkeMurthy}}}}
=
\mathcal O \left( \frac{|G|!\,m^4 \log m}{m!^3(|G|-3m)!} \right)
\]
as a bound for the runtime of our new algorithm. This is exactly the same bound as for the algorithm in \cite{HedtkeMurthy2011}. But as the results in Table~\ref{tab:runtime} show, the real runtime of our new algorithm highly depends on $m$ and the \emph{structure} of the group, whereas the real runtime of the old algorithms seems only to depend on $m$ and the \emph{size} of the group.

\subsection{Managing the (Parallel) Computation on a (Super-) Computer}

To compute the results (next section) for the search of \tens{5,5,5} TPP triple we used a supercomputer (a cluster with Sun Grid Engine) at the Martin-Luther-University Halle-Wittenberg. The computations (and their management) took several months. The number of $b_S^*$'s can be computed with
\begin{align*}
\text{\# of $b_S^*$'s} &= |\{(x_1,x_2,x_3,x_4)\in\mathbb N^4 ~:~ 1 \leq x_1 < x_2 < x_3 <x_4 \leq |G|\}|\\
&= \sum_{x_1=1}^{|G|-3}
\sum_{x_2=x_1+1}^{|G|-2}
\sum_{x_3=x_2+1}^{|G|-1}
\sum_{x_4=x_3+1}^{|G|} 1
=\frac{1}{24}(|G|^4-6|G|^3+11|G|^2-6|G|).
\end{align*}
The number of $b_S^*$'s for all groups in the Tables \ref{tab:c1cand} and \ref{tab:c2cand} can be found in Table~\ref{tab:NumberRowOnes}. We implemented the search algorithm with the optional arguments $\mi{startrow}$ and $\mi{numberOfRowOneTests}$ to realize a rudimentary parallelization: With an easy script we construct the set of all possible $b_S^*$'s and divide it into subsets of size 1{,}000 resp.\ 10{,}000. Now we start $(\text{\# of $b_S^*$'s})/1{,}000$ resp.\ $(\text{\# of $b_S^*$'s})/10{,}000$ independent jobs on a cluster, each with a different $\mi{startrow}$ that has to check $\mi{numberOfRowOneTests}=1{,}000$ resp.\ $\mi{numberOfRowOneTests}=10{,}000$ of the $b_S^*$'s. It is clear that even with an optimized search algorithm this is an immense amount of work. It follows right from that fact, that we dealt with tricks like going from the matrix representation $\bs C$ to the vector representation $\mi{marked}$ to get a sufficient speed-up to solve the \tens{5,5,5} problem.
\begin{table}
\centering
\begin{tabular}{r|rrrrrr}
$|G|$           & 48        & 52        & 54        & 55        & 56        & 57\\\hline
\# of $b_S^*$'s & 178{,}365 & 249{,}900 & 292{,}825 & 316{,}251 & 341{,}055 & 367{,}290\\
\multicolumn{1}{r}{~}\\[-0.5em]
$|G|$           & 60        & 63        & 72        & 78            & 80            & 88\\\hline
\# of $b_S^*$'s & 455{,}126 & 557{,}845 & 971{,}635 & 1{,}353{,}275 & 1{,}502{,}501 & 2{,}225{,}895
\end{tabular}
\caption{Number of $b_S^*$'s for all groups in the Tables \ref{tab:c1cand} and \ref{tab:c2cand}.}
\label{tab:NumberRowOnes}
\end{table}

\subsection{Results}
Our search for \tens{5,5,5} TPP triples in all groups of the C1 list showed, that no group can realize $5\times 5$ matrix multiplication with less than 100 scalar multiplications with the group-theoretic framework by Cohn and Umans. This continues the results \cite[Theorem~7.3]{HedtkeMurthy2011} of two of the authors who showed the same statement for $3 \times 3$ and $4 \times 4$ matrix multiplication.\par

\section{How to Construct a Matrix Multiplication Algorithm from a TPP Triple?}
\noindent As the results show, we were not able to find a group $G$ that realizes \tens{5,5,5} with $\underline{R}(G)<100$. But the groups in the C2 list could realize \tens{5,5,5} with less than 125 scalar multiplication, because $\underline{R}(G)<124$. This section shows a strategy to search for a nontrivial $5\times 5$ matrix multiplication algorithm in the C2 list.

Consider the case, that we found a \tens{5,5,5} TPP triple $(S,T,U)$ in a group $G$ of the C2 list. We only know that $\underline{R}(G)<125$, so we don't know if this leads to a nontrivial matrix multiplication algorithm. It could require 125 scalar multiplications or more. To construct the algorithm induced by the given TPP triple we have to construct the embeddings $\bs A \mapsto e_{\bs A}$ and $\bs B \mapsto e_{\bs B}$ of the matrices $\bs A=[a_{s,t}]$ and $\bs B=[b_{t,u}]$ in $\mathbb C[G]$:
\begin{gather}\label{eq:EmbInCG}
a_{s,t} \mapsto a_{s,t}s^{-1}t, \quad
b_{t,u}\mapsto b_{t,u}t^{-1}u \qquad
\text{for all } s\in S, t\in T, u\in U.
\end{gather}
The next step is to apply Wedderburn's structure theorem:
\begin{gather}\label{eq:EmbWed}
\mathbb C[G]
\cong
\mathbb C^{d_1\times d_1}
\times
\mathbb C^{d_2\times d_2}
\times
\cdots
\times
\mathbb C^{d_\ell\times d_\ell},
\end{gather}
where $d_1,\ldots,d_\ell$ are the character degrees of $G$. The given matrices $\bs A$ and $\bs B$ are now represented by $\ell$-tuples of matrices $e_{\bs A} \mapsto (\bs A_1,\ldots,\bs A_\ell)$ and $e_{\bs B} \mapsto (\bs B_1,\ldots,\bs B_\ell)$. The last step is easy: just use the best known algorithms to compute the products $\bs A_i \bs B_i$ or \emph{try to make use of the structures (e.g., symmetries, zero entries, \ldots) in $\bs A_i$ and $\bs B_i$ to find even better algorithms for the small products $\bs A_i \bs B_i$}. The back transformation works as in Equation~\eqref{eq:EmbInCG} but in the other direction.

Note that it could be possible to use the structure of the zero entries in $\bs A_i$ and $\bs B_i$: There is space for $d_i^2$ entries in each small matrix. Over all small matrices together we have enough space for $d_1^2+\cdots+d_\ell^2=|G|$ elements. But we only need space for $|S|\cdot |T|$ resp.\ $|T|\cdot |U|$ elements.

The key questions for future research are:
\begin{enumerate}
\item[(Q1)] Are there different embeddings \eqref{eq:EmbWed}, in the sense that they lead to different structures (pattern of zeros or other types) in the small matrices?
\item[(Q2)] Does the number $M(e)$ of multiplications needed to compute the product in \eqref{eq:EmbWed} depend on the embedding $e$?
\item[(Q3)] If so, we can bound $R(G)$ by $\min_e M(e)$. How many embeddings $e$ are there and how easy is it to compute $\min_e M(e)$?
\end{enumerate}

\begin{ex}
Consider the alternating group $A_5$ on five elements. The character degree pattern is $(1^1,3^2,4^1,5^1)$ and so
\[
\mathbb C [A_5]
\cong
\mathbb C \times 
\mathbb C^{3 \times 3} \times 
\mathbb C^{3 \times 3} \times 
\mathbb C^{4 \times 4} \times 
\mathbb C^{5 \times 5}.
\]
We know that $A_5$ realizes \tens{5,5,5}. There is place for 60 elements in the embedding $e_{\bs A}\in \mathbb C[A_5]$ of a $5 \times 5$ matrix $\bs A$ with $25$ elements. The same for $e_{\bs B}$. So we have to embed at most $|S^{-1}T \cup T^{-1}U|\leq |S^{-1}T| + |T^{-1}U| - 1 \leq 25 + 25-1=49$ elements into a space of $|A_5|=60$ elements. Assume that we can \emph{fill} the lower dimensional parts of the right hand side of \eqref{eq:EmbWed} \emph{completely}. Thus, only $49-1^2-3^2-3^2-4^2=14$ elements of the small matrices in $\mathbb C^{5 \times 5}$ are non-zero. Therefore it could be possible, that $A_5$ induces a nontrivial matrix multiplication algorithm: For the first \enquote{complete} parts we need $R(1)+2R(3)+R(4)=96$ scalar multiplications. We have 28 scalar multiplications left to compute the product of $\bs A_5\bs B_5$ to beat 125 scalar multiplications.
\end{ex}

\begin{ex}
The symmetric group $G:=S_3$ on three objects realizes $\tens{2,2,2}$ via the TPP triple $S=\{s_1=1_G,s_2=(1,2)\}$, $T=\{t_1=1_G,t_2=(1,3)\}$, $U=\{u_1=1_G,u_2=(2,3)\}$. We identify $A_{ij}$ with $A_{s_i,t_j}$ and $B_{jk}$ with $B_{t_j,u_k}$. The transformation into $\mathbb C[G]$ results in
\begin{align*}
c_1&:=a_{11}1_G + a_{12}(1,3) + a_{21}(1,2) + a_{22}(1,3,2),\\
c_2&:=b_{11}1_G + b_{12}(2,3) + b_{21}(1,3) + b_{22}(1,3,2).
\end{align*}
The character degree pattern of $S_3$ is $(1^2,2^1)$, so
$\mathbb C[G] \cong \mathbb C \times \mathbb C \times \mathbb C^{2\times 2}$.
To construct the map $f\colon \mathbb C[G] \to \mathbb C \times \mathbb C \times \mathbb C^{2\times 2}$, we follow \cite[Example~13.37]{Buergisser1997}. The irreducible representations of $S_3$ are
\begin{enumerate}
\item The trivial representation $\tau\colon S_3 \to \mathbb C$, $g \mapsto 1$.
\item The alternating representation $\alpha\colon S_3\to \mathbb C$, $g\mapsto \sgn(g)$.
\item The representation $\rho\colon S_3 \to \mathbb C^{2 \times 2}$, $(2,3)\mapsto \left[\begin{smallmatrix}
1 & -1\\ 0 & -1
\end{smallmatrix}\right]$, $(1,2,3)\mapsto \left[\begin{smallmatrix}
0 & -1\\ 1 & -1
\end{smallmatrix}\right]$.
\end{enumerate}
Thus, we conclude
\[
f\Big( \sum\nolimits_{g\in S_3} \lambda_gg \Big) = \Big( \sum\nolimits_{g\in S_3} \lambda_g \tau(g), ~ \sum\nolimits_{g\in S_3} \lambda_g \alpha(g), ~ \sum\nolimits_{g\in S_3} \lambda_g \rho(g) \Big).
\]
It follows that
\begin{align*}
f(c_1)&= \left(a_{11} {+} a_{12} {+} a_{21} {+} a_{22}, ~ a_{11} {+} a_{22} {-} a_{12} {-} a_{21}, ~ \begin{bmatrix}
a_{11} {-}a_{22} {-}a_{12}& a_{21}{+}a_{22}\\
a_{21} {-}a_{22} {-}a_{12}& a_{11}{+}a_{12}
\end{bmatrix}\right),\\
f(c_2) &= \left(b_{11} {+} b_{12} {+} b_{21} {+} b_{22}, ~ b_{11} {+} b_{22} {-} b_{12} {-} b_{21}, ~ \begin{bmatrix}
b_{11} {+} b_{12} {-} b_{21} {-} b_{22} & b_{22} {-} b_{12}\\
-b_{21}{-}b_{22} & b_{11} {-} b_{12} {+} b_{21}
\end{bmatrix}\right).
\end{align*}
In $\mathbb C$ we compute the product with 1 multiplication. In $\mathbb C^{2\times 2}$ we can use Strassen's algorithm with 7 multiplications. Therefore, we need 9 multiplications to calculate $f(c_1)f(c_2)$.
\end{ex}

This method provides a way to construct the multiplication algorithm induced by a given TPP triple. If it works (that means if one can answers questions (Q1), (Q2) and (Q3)), we find \emph{new best} or at least \emph{nontrivial} matrix multiplication algorithms for matrices of small dimension. Another approach to multiply matrices with a given TPP triple can be found in Gonzalez-Sanchez et al.\ \cite{Gonzalez-Sanchez2009}. But as far as we know, this approach doesn't construct the matrix multiplication algorithm itself.

\section{Conclusions}
\noindent From our point of view there are five open key questions or ideas one could use for future work.\par 

The first two are obviously the \tens{5,5,5} search in the C2 list, together with a practicable method to construct a matrix multiplication algorithm out of a given TPP triple. And C1-like searches for \tens{6,6,6} matrix multiplication algorithms and higher.\par

Is it easy and efficient to implement a search algorithm that does use products of quotients sets like in Theorem~\ref{thm:HedtkeMurthy}?\par 

Is there a constructive algorithm for TPP triples of a given \emph{type} \tens{n,p,m}?\par

As far as we know, the smallest example for a non-trivial matrix multiplication realized by the group-theoretic framework by Cohn and Umans is \tens{40,40,40}. The group $G=C_n^3 \wr C_2$ realizes \tens{2n(n-1),2n(n-1),2n(n-1)} with the rank $R(G)=2|G|-T(G)=4n^6-\frac12(n^6+3n^3)=\frac12n^3(7n^3-3)$, see \cite[Section~2]{Cohn2005} for details. Thus, for $n=5$ it realizes $40 \times 40$ matrix multiplication with $54{,}500$ scalar multiplications. This is way better than the naive matrix multiplication algorithm with $40^3=64{,}000$ scalar multiplications. On the other hand this is not a good result at all: Using $R(40)=R(2^3\cdot 5)\leq R(2)^3R(5)\leq 7^3\cdot 100=34{,}300$ we get an even better algorithm. The best known upper bound for the number of scalar multiplications in this case is
\[
\frac{n^3+12n^2+11n}{3}=\frac{40^3+12\cdot 40^2 + 11 \cdot 40}{3}= 27{,}880
\]
by \cite[Proposition~2]{Drevet2011}. Maybe our new algorithm can help to find a minimal working example for a non-trivial matrix multiplication algorithm realized with the group-theoretic framework by Cohn and Umans.

\end{document}